\documentclass{article}

\usepackage[utf8]{inputenc}
\usepackage{amssymb,amsthm,amsmath,enumerate}
\usepackage{graphicx,tikz}
\def\K{\mathbb{K}}
\def\T{\mathbb{T}}
\def\R{\mathbb{R}}
\def\I{ \textit{I} }
\def\II{ \textit{II} }
\def\III{ \textit{III} }
\def\padicdisc{\Delta}
\def\defi{\emph}
\def\bigvariables{X_1^{\pm 1},\ldots,X_d^{\pm 1}}
\def\smallvariables{x_1^{\pm 1},\ldots,x_d^{\pm 1}}
\newcommand{\plano}[2]{\ensuremath{H_{#1,#2}}}
\newcommand{\planopn}{\ensuremath{H_{p,n}}}
\newcommand{\planouniv}[1]{\ensuremath{\Delta(0)_{#1}}}
\newcommand{\discn}[1]{\ensuremath{\Delta_{#1,n}}}
\newcommand{\disc}[2]{\ensuremath{\Delta_{#1,#2}}}
\DeclareMathOperator\rad{rad}

\newtheorem{theorem}{Theorem}[section]
\newtheorem{lemma}[theorem]{Lemma}
\newtheorem{corollary}[theorem]{Corollary}
\theoremstyle{definition}
\newtheorem{definition}[theorem]{Definition}
\newtheorem{example}[theorem]{Example}

\newtheorem{remark}[theorem]{Remark}

\title{The tropical discriminant in positive characteristic}
\author{Luis Felipe Tabera}

\begin{document}
\maketitle

\begin{abstract}
We study singularities in tropical hypersurfaces defined by a valuation over a 
field of positive characteristic. We provide a method to compute the set of 
singular points of a tropical hypersurface in positive characteristic and the 
p-adic case. This computation is applied to determine all maximal cones of the 
tropical linear space of univariate polynomials of degree $n$ and characteristic 
$p$ with a fixed double root and the fan of all tropical polynomials that have 
$0$ as a double root independently of the characteristic. We also compute, by 
pure tropical means, the number of vertices, edges and 2-faces of the Newton 
polytope of the discriminant of polynomials of degree $p$ in characteristic $p$.
\end{abstract}

\section{Introduction}
Given $A\subseteq \mathbb{Z}^d$ a finite subset, there is a close relation 
between the theory of $A$-discriminants and coherent subdivisions and the 
secondary polytope of $A$. We refer to \cite{GKZ-book} for a basic reference on 
this relation. The combinatorial nature underlying the $A$-discriminant is more 
apparent computing the tropical discriminant of the support $A$, 
\cite{tropical-discriminant, 
singular-tropical-hypersurfaces,Tesis-master-Ochse}. However, this study is 
usually restricted to the case of characteristic $0$. In this paper, we extend 
the notion of tropical singularity in a hypersurface introduced in 
\cite{singular-tropical-hypersurfaces} to the characteristic $p$ and the 
$p$-adic case, with the aim that this study will help understanding the 
reduction of the $A$-discriminant mod $p$.

With this idea in mind, let $\K$ be an algebraically closed field with a 
valuation $v:\K^*\rightarrow \T\subseteq \mathbb{R}$. Let $k$ be the residue 
field and let $p$ be a prime number. There are three possibilities for the 
characteristics of $\K$ and $k$.

\begin{itemize}
 \item $char(\K)=char(k)=0$, (equi)characteristic zero.
 \item $char(\K)=char(k)=p$, (equi)characteristic $p$.
 \item $char(\K)=0$, $char(k)=p$, $p$-adic case.
\end{itemize}

If $V\subseteq (\K^*)^d$ is an algebraic variety of dimension $n$ in the torus, 
its tropicalization is the closure in $\R^n$ of the image of $V$ taking the 
valuation component-wise,
\[trop(V)=\overline{\{(v(a_1),\ldots,v(a_d))\in \R^d\ |\ (a_1,\ldots, a_d)\in 
V}\}\]
These varieties are polyhedral complexes of dimension $n$ in $\R^d$ and are
the base of tropical geometry. In this paper, we are 
interested in the case that $\K$ is a field of characteristic $p$ and 
$V=\Delta_A$ is the $A$-discriminant, the set of polynomial of support $A$ 
having a double root in $(\K^*)^d$.

In many cases, tropical geometry do not depend in the characteristic of the 
underlying field. For instance, Kapranov's theorem does not depend on the 
characteristic \cite{Lifting-Constr,Tesis-Speyer,Tesis-Tab}.
Also, if we fix $d+1$-supports $A_0,\ldots, A_{d}$ 
from the results of \cite{Sturmfels-polytope_resultant,Resultantes-trop} it 
follows that the $(A_0,\ldots, A_d)$ tropical resultant does not depend on the 
characteristic of the field. On the other hand, it is known that the tropical 
Grassmannian --understood as the image of the Grassmannian under a valuation--
does depend on the characteristic of the ground field \cite{tgrassmanian}. This 
is also the case of the discriminant. For instance, if $N_{p,n}$ is the Newton 
polytope of the discriminant of a univariate polynomial of degree $n$ in 
characteristic $p$, it is well known that $N_{n,0}$ is combinatorially a 
$(n-1)$-hypercube, while we prove (See Corollary~\ref{cor:tria_y_cuad}) that 
the 2-faces of $N_{p,n}$ are quadrangles or triangles.

The paper is structured as follows. In Section~\ref{sec:preliminares}, we 
introduce the basic notions and the notation. In 
Section~\ref{sec:singularidades} the set of singular points of a hypersurface 
is computed using pure tropical techniques, thus, giving a purely combinatorial 
method to decide if a polynomial is in the $A$-discriminant in characteristic 
$p$ or the $p$-adics. Section~\ref{sec:planopn} is the main section, we study
the (tropical) linear space of all tropical univariate polynomials of 
degree $n$ having a double root in characteristic $p$, some results on the 
tropical discriminant in characteristic $p$ and we also describe the set of 
tropical polynomials that are singular independently of the characteristic. In 
\ref{sec:padics}, we make a brief comment on the case of the $p$-adics.

\section{Preliminaries}\label{sec:preliminares}

Let $\K$ be an algebraically closed field. Let $\mathbb{Q}\subseteq \T\subseteq 
\R$ be a subgroup of the reals and $v:\K^*\rightarrow \T$ a nontrivial 
valuation.
We will 
denote by capital letters $A,B,X$ the elements and variables in $\K$ and by 
lower case $a,b,x$ the elements and variables in $\T$. Given a Laurent 
polynomial $F=\sum_{i\in A} A_ix^i\in \K[\bigvariables]$, $i=(i_1,\ldots, 
i_d)$, $X=X_1\cdots X_d$ we call the \defi{tropicalization} of $F$ to the 
formal tropical polynomial $f=\oplus_{i\in A} a_ix^i$, with $a_i=v(A_i)$, $i\in 
A$. The function associated to $f$ is the piecewise-affine function
\[f(b)=\min_{i\in A}\{a_i +\langle i,b\rangle \}\]
Different polynomials may define the same function, for instance $0\oplus 
0x\oplus 0x^2$ and $0\oplus 1x\oplus 0x^2$. $f(b)$ is the expected valuation of 
$F(B)$ for $B$ an element such that $v(B)=b$. But there may be some $B$ such 
that $f(v(B))\leq v(F(B))$. This can only happen if $b$ is a \defi{tropical 
root} of $f$. That is, if the value $f(b)$ is attained (at least) at two 
different monomials. We denote by $\mathcal{T}(f)$ the set of tropical roots of 
$f$.
\[\mathcal{T}(f)=\{b\in \R^n\ |\ \exists i,j\in A, f(b)=a_i+\langle 
b,i\rangle=a_j+\langle b,j\rangle\}\]

If $f$ is the tropicalization of $F$ and $V=\mathcal{V}(F)\subseteq (\K^*)^d$ 
is the hypersurface defined by $F$ then, by the fundamental theorem of tropical 
geometry
\[\mathcal{T}(f)=\overline{\{(v(B_1),\ldots,v(B_d))\ |\ (B_1,\ldots, B_d)\in 
V\}}.\]
See for instance \cite{Lifting-Constr, Tesis-Tab}.

\begin{definition}
Let $f=\oplus_{i\in I} a_ix^i\in \T[\smallvariables]$ be a Laurent tropical 
polynomial in $d$ variables. We say that $f$ is a \defi{singular polynomial} 
(with respect to a valuation $v:\K\rightarrow \T$) and that 
$b=(b_1,\ldots,b_d)$ is a \defi{singular point of} $f$ if there exists an 
algebraic counterpart polynomial $F=\sum_{i\in I} A_iX^i\in 
\K[\bigvariables]$, $B=(B_1,\ldots,B_d)\in (\K^*)^d$ with $v(A_i)=a_i$, 
$v(B_i)=b_i$ and such that $B$ is a singular point of the hypersurface defined 
by $F$.
\end{definition}

\begin{example}
Let $f=0\oplus 0 \oplus 0x^2$ and $g=0\oplus 1x \oplus 0x^2$. Then 
$\mathcal{T}(F)=\mathcal{T}(G)=\{0\}$. If we are working in the characteristic 
zero case, then $f$ is singular, since it is the tropicalization of 
$X^2-2X+1=(X-1)^2$. Moreover, $g$ cannot be singular, since if $G=AX^2+BX+C$ is 
a polynomial with tropicalization $g$, then $v(A)=v(C)=0$, $v(B)=1$ and 
$v(B^2-4AC)=0$, so $B^2-4AC\neq 0$ and $G$ does not have a double root.

On the other hand, if we now work with a $2$-adic valuation, $v_2(2)=1$, then 
$g$ is the tropicalization of $X^2-2X+1$ and $g$ is singular. $f$ cannot be 
singular in this case, since if $F=AX^2+BX+C$ has $v_2(A)=v_2(B)=v_2(C)=0$, 
then $v_2(B^2-4AC)=\min\{0,2\}=0$ and the discriminant does not vanish.
\end{example}

We now define the \defi{tropical Euler derivatives} introduced in 
\cite{singular-tropical-hypersurfaces}, this is the main tool we use to deal 
with tropical singularities.

\begin{definition}
Let $F=\sum_{i\in I}A_iX^i\in \K[\bigvariables]$. Let 
$L=B_0+B_1X_1+\ldots+B_dX_d\in \mathbb{Z}[\bigvariables]$ be an linear 
polynomial defined over the integers. We define the \defi{Euler derivative with 
respect to $L$} to
\[\frac{\partial F}{\partial L} = B_0F+\sum_{j=1}^d B_jX_j\frac{\partial 
F}{\partial X_j} = \sum_{i\in I} L(i)A_iX^i\in \K[\bigvariables]\]
Analogously, let $f\in \sum_{i\in I}a_ix^i\in \T[\smallvariables]$. The 
\defi{Euler derivative with respect to $L$} is
\[\frac{\partial f}{\partial L} = \bigoplus_{\substack{i\in I\\L(i)\neq 0}} 
(v(L(i))+a_i)x^i\in \T[\smallvariables]\]
Note that the definition of partial Euler derivative depends on the specific 
valuation of the field. Note also that if $f$ is the tropicalization of $F$, 
$v(A_i)=a_i$, then $v(L(i)A_i)=v(L(i))+a_i$ and $\frac{\partial f}{\partial L}$ 
is the tropicalization of $\frac{\partial F}{\partial L}$.
\end{definition}

\begin{remark}
Let $f$ be a tropical polynomial with support $A$, let 
$L=b_0+b_1x_1+\ldots+b_dx_d\in 
\mathbb{Z}[\smallvariables]$ and $\gcd(b_0,\ldots,b_d)=1$. Then taking the 
Euler derivative in $f$ has an easy geometric interpretation.
\begin{enumerate}
\item Characteristic $0$, we eliminate from the support of $f$ the monomials in 
the hyperplane $\{L=0\}$, the coefficients remain unchanged.
\item Characteristic $p$, we eliminate from the support of $f$ all the 
monomials lying at lattice distance $r\equiv 0\mod p$ from the hyperplane 
$\{L=0\}$, 
the rest of the coefficients remain unchanged.
\item $p$-adic case, we eliminate from the support of $f$ the monomials lying 
in the hyperplane $L=0$. If a monomial $i_0$ lies at lattice distance $r\equiv 
0\mod p$ from $L=0$, we add $v_p(r)$ to the corresponding coefficient $a_{i_0}$ 
of $f$.
\end{enumerate}
\end{remark}

\begin{example}
Let $f=a_0\oplus a_1x\oplus a_2x^2\oplus a_3x^3\oplus a_4x^4\oplus a_5x^5$. Let 
$L=x-4$. Then, the Euler derivative of $f$ with respect to $L$ is:
\begin{itemize}
 \item In the characteristic $0$ case. In the characteristic $p$ case or 
$p$-adic case with $p>3$, $\frac{\partial f}{\partial L}=a_0\oplus a_1x\oplus 
a_2x^2\oplus a_3x^3\oplus a_5x^5$.
 \item If the characteristic of $\K$ is $p=2$, $\frac{\partial f}{\partial 
L}=a_1x\oplus 
a_3x^3\oplus a_5x^5$.
 \item If the characteristic is $p=3$, $\frac{\partial f}{\partial L}=a_0\oplus 
a_2x^2\oplus a_3x^3\oplus a_5x^5$.
 \item In the $2$-adic case, $\frac{\partial f}{\partial L}=(2+a_0)\oplus 
a_1x\oplus (1+a_2)x^2\oplus a_3x^3\oplus a_5x^5$.
 \item In the $3$-adic case, $\frac{\partial f}{\partial L}=a_0\oplus 
(1+a_1)x\oplus a_2x^2\oplus a_3x^3\oplus a_5x^5$.
\end{itemize}
\end{example}

\section{Singularities in tropical hypersurfaces}\label{sec:singularidades}

\begin{lemma}\label{lem:is_a_fan}
If the tropical discriminant $\Delta_{p,A}$ of polynomials of support $A$ in 
characteristic $p$ is non-empty, then it is a rational polyhedral fan of pure 
dimension in $\R^{|A|}$.
\end{lemma}
\begin{proof}
Let $\tilde{\Delta_{p,A}}$ be the $A$-discriminant in characteristic $p$. Since 
it is parametrizable, $\tilde{\Delta_{p,A}}$ is absolutely irreducible. It 
follows from \cite{Bieri-Groves} that $\Delta_{p,A}$ is a rational polyhedral 
complex of the same dimension as $A$. Moreover, since $\tilde{\Delta_{p,A}}$ is 
a variety defined over the prime field $\mathbb{Z}/(p)$, then we are in a 
constant coefficient case, the initial ideal $in_{w}(I)=in_{tw}(I)$, for every 
$w\in \R^n$ and $t\in \R$. It follows from \cite{Lifting-Constr} that the
the cells of $\Delta_{p,A}$ are cones and $\Delta_{p,A}$ is a rational 
polyhedral fan.
\end{proof}

\begin{theorem}[\cite{singular-tropical-hypersurfaces}]
Let $f=\bigoplus_{i\in I} a_i  x^i$ be a tropical polynomial with support
$A$. Let $q\in \mathcal{T}(f)$ be a point in the hypersurface defined by $f$.
Then, $q$ is a singular point of $\mathcal{T}(f)$ if and only if $q\in
\mathcal{T}(\frac{\partial f}{\partial L})$ for all $L$. 

Thus, $f$ defines a singular tropical hypersurface if and only if \[\bigcap_{L}
\mathcal{T}\left(\frac{\partial f}{\partial L}\right)\neq \emptyset.\] This 
intersection can be given by a finite number of Euler derivatives of $f$.
\end{theorem}
\begin{proof}
The proof given in \cite{singular-tropical-hypersurfaces} is written for the 
characteristic zero case, but it works in any case. The result in the general 
case follows from the fact that given a linear space $V\subseteq (\K^*)^d$, a 
tropical basis of $V$ is given by the linear polynomials vanishing in $V$ with 
minimal support. These polynomials form a tropical basis independently of the 
characteristic of the field $\K$ nor the valuation $v$.
\end{proof}

\begin{example}
Let $f$ be the tropical polynomial $f=0\oplus 0x^2\oplus 0y^2\oplus 
0x^2y^2\oplus 
1x^3$. Then, it is easy to check that $(0,0)$ is a singularity of $f$ in 
characteristic different from $2$, but that $f$ is not singular in 
characteristic $2$, because $\frac{\partial f}{\partial x-y}=x^3$ that has no 
tropical root.
\end{example}

\begin{corollary}\label{cor:solo_necesito_p}
Let $f=a_0\oplus a_1x\oplus\ldots\oplus a_nx^n\in \T[x]$ be a tropical 
univariate 
polynomial of degree $n$ in characteristic $p\neq 0$, let $a\in 
\mathcal{T}(f)$, then $f$ is singular with double root $a$ if and only if $a\in 
\mathcal{T}(\frac{\partial f}{\partial x-i})$, $i=0,\ldots, p-1$.
\end{corollary}
\begin{proof}
If $L=ax+b$, then $\frac{\partial f}{\partial L}$ is, up to multiplication by a 
constant, equal 
to $f$ or one of the Euler derivatives $\frac{\partial f}{\partial x-i}$, 
$i=0,\ldots, p-1$.
\end{proof}

The case of characteristic zero is well understood. The tropical discriminant 
is combinatorially dual to the Newton polytope of the discriminant, see,
as a starting point, 
\cite{tropical-discriminant}, \cite{GKZ-polytope-resultant}, \cite{GKZ-book}. 
Here, we will follow the approach of \cite{sing-fixed-point}, 
\cite{tropical-surf-singularities}. We study the univariate polynomials of 
degree $n$ that have a singular point in $0\in \T$. Any other case can be 
reduced to this situation.

\section{The linear spaces $\planopn$}\label{sec:planopn}

In this section we follow the ideal of \cite{sing-fixed-point}. We fix the 
tropical point $0$ and compute all polynomials with $0$ as double root. The 
space obtained is a tropical linear space.

\begin{definition}
We define $\planopn$, with $p$ prime or zero, the (tropical) linear space 
of tropical polynomials of degree $n$ such that $0$ is a tropical double root.
If $\discn{p}$ is the set of tropical singular polynomials of degree $n$ in 
characteristic $p$ then \[\discn{p}=\bigcup_{c\in \T} \{f(cx)|f\in 
\planopn\}\]
\end{definition}

However, passing from $\planopn$ to $\discn{p}$ is not trivial. If one 
wants to use the techniques developed here to study $\discn{p}$, some knowledge 
of tropical polynomials with two tropical roots is needed. This is an open 
problem even in characteristic zero.

\begin{theorem}
$\plano{0}{n}$ is the set of tropical polynomials $f=\oplus_{i=0}^n a_ix^i$ 
such that the minimum of $\{a_i, 0\leq i\leq n\}$ is attained (at least) at
three different monomials $\{j,k,l\}$. It is a rational fan of codimension 2 in 
$\T^n$, the maximal cones can be marked by the monomials $\{j,k,l\}$ where the 
minimum is attained, so it has $\binom{n+1}{3}$ maximal cells.
\end{theorem}

\subsection{The case $p=2$}

\begin{theorem}
A tropical polynomial $f$ is in $\plano{2}{n}$ if and only if the minimum 
coefficient in the even monomials of $f$ is attained twice and in the odd 
monomials is attained also twice. Hence, if $f$ is of degree $2n-1$, $n\geq 2$ 
then there are $\binom{n}{2}^2$ maximal cones.
\end{theorem}
\begin{proof}
By Corollary~\ref{cor:solo_necesito_p} a tropical polynomial $f$ is singular at 
$0$ if 
and only if $f_0=\frac{\partial f}{\partial x}$ and $f_1=\frac{\partial 
f}{\partial x-1}$ have $0$ as tropical root. $f_0$ is the polynomial consisting 
on 
all odd monomials of $f$ and $f_1$ is the polynomial consisting on all even 
monomials of $f$.
\[f_0=\bigoplus_{i=0}^{\lfloor \frac{n-1}{2}\rfloor} a_{2i+1}x^{2i+1}
\quad f_1=\bigoplus_{i=0}^{\lfloor \frac{n}{2}\rfloor} a_{2i}x^{2i}\]
A maximal cone is characterized by the two monomials $\{i,j\}$ where $f_0$ 
attains its minimum and the two monomials $\{k,l\}$ where $f_1$ attains its 
minimum. Hence, in a polynomial of degree $2n-1$ there are $\binom{n}{2}^2$ 
maximal cells.
\end{proof}

\begin{theorem}
The tropical discriminant $\disc{2}{2n-1}$ in characteristic $2$ of a 
polynomial 
of degree $2n-1$ is in natural correspondence with the tropical resultant of 
two 
polynomials of degree $n-1$ in characteristic zero. In particular, the Newton 
polytope of $\disc{2}{2n-1}$ has $\binom{2n-2}{n-1}$ extremal vertices.
\end{theorem}

These two polynomials have a common tropical root if and only if
\[g_1 = \bigoplus_{i=0}^{\lfloor \frac{n}{2}\rfloor} a_{2i}y^{i},\quad 
g_0=\bigoplus_{i=0}^{\lfloor \frac{n-1}{2}\rfloor} a_{2i+1}y^{i}\]
have a common root. Hence, the polynomials in the tropical discriminant 
$\discn{2}$ is in bijective correspondence to the tropical resultant of two 
univariate polynomials of degree $\lfloor \frac{n}{2}\rfloor$ and $\lfloor 
\frac{n-1}{2}\rfloor$. Since the tropical resultant of two univariate 
polynomials does not depend on the characteristic 
\cite{Sturmfels-polytope_resultant, 
Resultantes-trop}, then 
the tropical discriminant in characteristic two is in natural bijective 
correspondence with the resultant of two polynomials of degree $\lfloor 
\frac{n}{2}\rfloor$ and $\lfloor \frac{n-1}{2}\rfloor$.

\subsection{The case $p>2$}
Let us study now the case of an odd primes $p$. By 
Corollary~\ref{cor:solo_necesito_p}, a polynomial $f$ is in $\planopn$ if 
and 
only if the minimum of the coefficients of $\frac{\partial f}{\partial x=i}$ is 
attained twice for $i=0,\ldots, p-1$. Let us now describe the combinatorial 
types of maximal cells in $\planopn$.

\begin{theorem}\label{teo:clasifica_tipos_maximales}
The set of maximal cells in $\planopn$ is exactly the following:
\begin{enumerate}
\item Type \I. Tropical polynomials where the minimum is attained at three 
different 
monomials $i,j,k$ that produce three different residues $\mod p$. Denote them 
by $\{i,j,k\}$
\item Type \II. Tropical polynomials where the minimum is attained at $i,j$ 
with the same residue $\mod p$. And, if we eliminate the monomials $r$, 
$r\equiv i\mod p$, 
then the minimum is attained in $k,l$ con $k\not\equiv l\mod p$. Denote them by 
$[\{i,j\},\{k,l\}]$
\item Type \III. Tropical polynomials where the minimum is attained at $i,j$ 
with the same class $\mod p$. And, if we eliminate the monomials $r$, $r\equiv 
i\mod p$, then 
the minimum is attained in $k,l$ with $k\equiv l\mod p$. Denote them by 
$\{\{i,j\},\{k,l\}\}$.
\end{enumerate}
\end{theorem}
\begin{proof}
It is clear that any polynomial that satisfies the description of the cells is 
in $\planopn$. For type \I polynomials and any Euler derivative, there is 
at 
least two monomials in $\{i,j,k\}$ where the minimum is attained. For type \II 
and \III. For any Euler derivative $x=s$, with $s\neq r$ the minimum is 
attained at $\{i,j\}$, while for $\frac{\partial f}{\partial x-r}$ the minimum 
is attained at $\{k,l\}$. We are also describing cells of codimension $2$, so 
we are dealing with maximal cells.

Let us now check that there is no other possibility. Let $f\in\planopn$, we 
will show that it must belong to the closure of one of these cells. Let $L_0$ 
be the set of monomials where the minimum is attained. Consider $L_0\mod 
p=\{[a]| a\in L_0\}$ the set of classes. If there are at least three different 
classes $[i],[j],[k]$ in $L_0\mod p$, then $f$ is in the closure of the type \I 
cell $\{i,j,k\}$. If $L_0\mod p = \{[r]\}$ is a single class. Then $L_0$ 
contains at least two monomials $\{i,j\}$. Now, since $f\in\planopn$, the 
minimum of the coefficients of $\frac{\partial f}{\partial x-r}$ has two be 
attained at least for two monomials $k,l$ hence $f$ is in the closure of a type 
\II or \III cell. Depending if $k\equiv l\mod p$ or not. Finally, it may happen 
that $L_0$ consists on exactly two different classes $\{[r],[s]\}$. Since the 
minimum of $\frac{\partial f}{\partial x-r}$ and $\frac{\partial f}{\partial 
x-s}$ is attained at least twice, there must be at least two monomials on each 
class $i,j\in L_0$ $[i]=[j]=[r]$ and $k,l\in L_0$, $[k]=[l]=[s]$. Thus, in this 
case $f$ is in the closure of a type \III cell.

Finally, we have to check that we are not counting twice a maximal cell.
\end{proof}

We now check that the notation of type \III cells is well chosen.

\begin{lemma}\label{lem:type_III_swap}
Let $[i]=[j]\neq [k]=[l]\mod p$. Then the cones $a_i=a_j\leq a_k=a_l$ and 
$a_k=a_l\leq a_i=a_j$ belong to the same maximal cone of type \III, 
$\{\{i,j\},\{k,l\}\}$.
\end{lemma}
\begin{proof}
The two cones meet in the common face $D=\{a_i=a_j=a_k=a_l\}$. In 
characteristic zero, this cone is adjacent to four maximal cones of 
$\plano{0}{n}$, namely $\{i,j,k\}$, $\{i,j,l\}$, $\{i,k,l\}$, $\{j,k,l\}$. 
However, none of these cones belong to $\planopn$, since $[i]=[j]\neq 
[k]=[l]\mod p$. The only perturbations of $D$ that still belong to $\planopn$ 
are precisely the cones $a_i=a_j\leq a_k=a_l$ and $a_k=a_l\leq a_i=a_j$. It 
follows that the maximal cone of $\planopn$ containing $D$ is $a_i=a_j$, 
$a_k=a_l$, that is the definition of the cone $\{\{i,j\},\{k,l\}\}$.
\end{proof}

\begin{remark}
The only cells in common of $\planopn$ and $\plano{0}{n}$ are type \I 
cells. The reason for separating type \II and type \III cells is that if 
$[\{i,j\},\{k,l\}]$ is a type \II cell, then the cone $a_k=a_l<a_i=a_j$ is not 
in $\planopn$. Moreover, in the proof of Theorem~\ref{teo:incidencias} we will 
see that if $\{\{i,j\},\{k,l\}\}$ the set of polynomials $a_i=a_j<a_k=a_l$ and 
$a_k=a_l<a_i=a_j$ (and the rest of monomial higher) belong to a maximal cone of 
$\planopn$.
\end{remark}

\begin{theorem}\label{teo:conteo_de_plano}
In $\plano{p}{pn-1}$ in characteristic $p>2$ there are:
\begin{enumerate}
 \item $\binom{p}{3}n^3$ facets of type \I
 \item $pn^2\binom{n}{2}\binom{p-1}{2}$ facets of type \II
 \item $\binom{p}{2}\binom{n}{2}^2$ facets of type \III
\end{enumerate}
Thus, for fixed $p$ the number of facets in $\plano{p}{m}$ is 
$\mathcal{O}(m^4)$.
\end{theorem}
\begin{proof}
If the polynomial is of degree $pn-1$ there are exactly $n$ monomials on each 
class $\mod p$. For type \I facets, we have to choose three classes 
$[i],[j],[k]$ and, on each class one monomial. For type \II facets 
$[\{i,j\},\{k,l\}]$ we first choose the class of $[i],[j]$ and then the two 
monomials. Now, we have to choose the classes of $[k]$, $[l]$ and one element 
on each class. Finally, for type \III facets, $\{\{i,j\},\{k,l\}\}$ we have 
just to choose the two classes $[i]$, $[k]$ and, on each class. Two monomials.
\end{proof}

We now describe the incidence of two maximal cones of $\discn{p}$. If $C_1$ and 
$C_2$ are two different maximal cones of $\discn{p}$ and the common face 
$D=C_1\cap C_2$ is of dimension $n-1$ in $\R^{n+1}$, then any polynomial $f$ 
will have two double roots $a_1$, $a_2$. The specific values of $a_1$ and $a_2$ 
depend on $f\in D$, but not the fact that $a_1=a_2$ or $a_1\neq a_2$.

\begin{theorem}\label{teo:incidencias}
Let $p>2$ be a prime. Let $C_1$, $C_2$ be two maximal cones of $\discn{p}$ 
meeting on a cone dimension $n-1$, $D=C_1\cap C_2$. Let $a_1$, $a_2$ the two 
tropical double roots of a generic polynomial $f\in D$. Then
\begin{enumerate}
 \item If $a_1\neq a_2$ then $C_1$ and $C_2$ can be of any type, $D$ is a face 
of exactly 4 maximal cones of $\discn{p}$.
 \item If $a_1=a_2$ then, up to order:
\begin{enumerate}
\item The only obstruction for the pair of types of $C_1$ and $C_2$ is \I-\III.
\item If both $C_1=\{i,j,k\}$ and $C_2=\{i,j,l\}$ are of type \I and the classes
$[i],[j],[k],[l]$ are pairwise different, then $D$ is adjacent to 4 maximal 
cones of 
$\discn{p}$. Otherwise $D$ is adjacent to exactly 3 maximal cones of 
$\discn{p}$.
\end{enumerate}
\end{enumerate}
\end{theorem}
\begin{proof}
By abuse of notation, if $i\in I$, the support of the polynomial, we will also 
write $i=(i,a_i)$ the corresponding point in the Newton diagram of the 
polynomial. Let $a_1\neq a_2$. We distinguish the following cases.\\
$\bullet$ Both $C_1$ and $C_2$ are of type \I, we are in the same situation as 
in characteristic 0. Since every 2-face of the newton polytope in discriminant 
zero is (combinatorially) a square, $D$ is adjacent to four different maximal 
cones.\\
$\bullet$ If $C_1=\{i,j,k\}$ and $C_2$ is $[\{l,m\},\{r,s\}]$. Without loss of 
generality $i\notin \{l,m,r,s\}$ (because $i,j,k$ are collinear). Also, 
$r\notin \{i,j,k\}$ (because $\{i,j,k\}$ lie on an edge of the Newton polytope 
of $f$ not parallel to $\overline{rs}$). We can either increase or decrease $i$ 
or $o$ to eliminate one of the tropical roots and passing to a maximal cone. 
Hence $S$ is also adjacent to four maximal cells. The same argument holds for 
two cells of type \I, \III.\\
$\bullet$ both $C_1$ and $C_2$ are of type \II or \III. Assume that 
$C_1=[\{i,j\},\{k,l\}]$, $C_2=[\{r,s\},\{t,u\}]$. Note that it is impossible 
that $\{i,j\}\subseteq \{r,s,t,u\}$. First, $\{r,s\}\neq \{i,j\}\neq \{t,u\}$, 
sine $\overline{ij}$ is not parallel to the lines 
$\overline{rs}||\overline{t,u}$. Second, it can not happen that $i\in \{r,s\}$ 
and $j\in \{t,u\}$, because in that case $[r]=[i]=[j]=[t] \mod p$ and this is 
impossible. Hence we may assume that $i\notin \{r,s,t,u\}$ and $r\notin 
\{i,j,k,l\}$ and we still have four possibilities to perturb this configuration 
into maximal cones of $\discn{p}$.

Now, assume that $a_1=a_2$, by dehomogenization, we may assume that 
$a_1=a_2=0$. We also distinguish different cases:\\
$\bullet$ $C_1$ and $C_2$ are of type \I. They are of the form $C_1=\{i,j,k\}$, 
$C_2=\{i,j,t\}$. If $[k]\neq [t]$, then the case is like the characteristic 
zero case, $D$ is adjacent to the four cells of type \I, $\{i,j,k\}$, 
$\{i,j,t\}$, $\{i,k,t\}$, $\{j,k,t\}$. But, if $[k]=[t]$, then $D$ is a face of 
the three maximal cones $\{i,j,k\}$, $\{i,j,t\}$ and $[\{k,t\},\{i,j\}]$. Note 
that $[i]\neq [j]$, so this cell is never of type \III.\\
$\bullet$ Types \I-\II. In this case, the two cells are $\{i,j,k\}$ and 
$[\{k,l\},\{i,j\}]$, so again, $D$ is a face of three maximal cells. 
$\{i,j,k\}$, $\{i,j,l\}$, $[\{k,l\},\{i,j\}]$.\\
$\bullet$ Types \I-\III. If $\{i,j,k\}$, $\{j,r\},\{k,s\}$ meet in $D$, we have 
five monomials with the same value at the tropical root, so $D$ is at least of 
codimension 3 in $\R^n$.\\
$\bullet$ Types \II-\II. Two cells of type \II can meet in two different ways.
\begin{enumerate}
 \item $[\{i,j\},\{k,l\}]$ and $[\{i,j\},\{k,m\}]$. $D$ is a face of these two 
type \II cells and the other may be the cell of type \II or \III 
$[\{i,j\},\{l,m\}]$ or $\{\{i,j\},\{l,m\}\}$ depending if $[l]\neq[m]$ or not.
 \item $[\{i,j\},\{k,l\}]$ and $[\{i,m\},\{k,l\}]$, $D$ is a face of three 
maximal cells. The previous two plus $[\{j,m\},\{k,l\}]$. Note that, in this 
case $[i]=[j]=[m]$, so it is not important if $a_i=a_j<a_m<a_k=a_l$ to decide 
if the polynomial is singular or not, because the only important derivative 
$\frac{\partial f}{\partial x-i}$ erases monomials $i,j,m$.
\end{enumerate}
$\bullet$ Types \II-\III. The two cells are of the form $[\{i,j\},\{k,l\}]$ 
and 
$\{\{i,j\},\{l,m\}\}$. This is one of the cases studied above. The codimension 
one cell is adjacent to these two facets plus $[\{i,j\},\{k,m\}]$.\\
$\bullet$ Type \III-\III. Two different cells of type \III meet, they must be 
of one of the following cases:
\begin{enumerate}
\item $\{\{i,j\},\{k,l\}\}$ and $\{\{i,j\},\{k,m\}\}$ so there is again only 
another facet adjacent to this cell, $\{\{i,j\},\{l,m\}\}$.
\item $\{\{i,j\},\{k,l\}\}$ and $\{\{i,j\},\{n,o\}\}$, with $[k]\neq [o]\neq 
[i]$. The other adjacent cell is $\{\{k,l\},\{n,o\}\}$
\end{enumerate}
\end{proof}

\begin{corollary}\label{cor:tria_y_cuad}
Let $p>2$ be a prime. If $n<p$ then $\discn{p}=\disc{0}{n}$. If $n\geq p$ then 
the $2$-faces of the Newton polytope of $\discn{p}$ are quadrangles or 
triangles.
\end{corollary}
\begin{proof}
If $n<p$, then the Euler derivative $\frac{\partial f}{\partial x-i}$ is the 
same in characteristics $0$ and $p$, hence $\discn{p}=\discn{p}$. If $n\geq p$ 
the 2-faces of $\discn{p}$ are combinatorially dual to the cones $D$ in 
$\discn{p}$ of codimension 2. By Theorem~\ref{teo:incidencias} $D$ is a face of 
four or three maximal cones. Hence the 2-faces can only be quadrangles and 
triangles.
\end{proof}

\begin{theorem}
The Newton polytope of $\disc{p}{p}$ has $2^{p-1}-1$ vertices, 
$(p-1)(2^{p-2}+p/2-2)$ edges, $\binom{p-1}{2}(2^{p-3}-1)$ quadrangles and 
$\binom{p}{3}$ triangles.
\end{theorem}
\begin{proof}
We give a pure tropical proof of this theorem. First, note that in 
$\disc{p}{p}$ there are only cells of type \I and \II. Let $N_0$ and $N_p$ be 
the Newton polytopes of $\disc{0}{p}$ and $\disc{p}{p}$ respectively. The 
vertices of $N_q$ correspond to the cones in the complement of $\disc{q}{p}$, 
that is, combinatorial types of polynomials without double roots. Edges 
correspond to codimension 1 cones in $\disc{q}{p}$ and 2-faces of $N_q$ with 
codimension 2 cones in $\disc{q}{p}$. The Newton polytope $N_0$ of 
$\disc{0}{p}$ is the secondary polytope of $\{0,\ldots, p\}$, combinatorially a 
hypercube of dimension $p-1$ \cite{GKZ-polytope-resultant}. Its vertices 
correspond to the 
triangulations of the set $[p]=\{0,\ldots, p\}$. Note that, reducing $\mod p$, 
the only modifications appear around the vertex $V_0=(A_0A_p)^{p-1}$, 
corresponding to the triangulation $[0,p]$. This vertex is not in $N_p$, since 
$f\in C_0$ is not enough information to decide if $f\in \disc{p}{p}$ or not. 
The common cones of type \I in $\disc{0}{p}$ and $\disc{p}{p}$ are those of 
type $\{i,j,k\}$ except $\{0,i,p\}$, $1\leq i\leq p-1$ that are not cones in 
$\disc{p}{p}$. These are the edges in $N_0$ linking $V_0$ with the vertices 
$[0,i,p]$ of $N_0$. Hence, $p-1$ edges disappear reducing $\mod p$. On the 
other hand, the new edges corresponding to type \II cells $[\{0,p\},\{i,j\}]$ 
appear in $N_p$. Thus we have 
$(p-1)2^{p-2}-(p-1)+\binom{p-1}{2}=(p-1)(2^{p-2}+p/2-2)$ edges in $N_p$. In 
$N_0$ there are $\binom{p-1}{2}(2^{p-3})$ quadrangles, from these 
$\binom{p-1}{2}$ are adjacent to $V_0$ and no new quadrangle appear. Hence 
$N_p$ has $\binom{p-1}{2}(2^{p-3}-1)$ quadrangles. Finally, the number of 
triangles correspond to codimension 2 cones in $\disc{p}{p}$ that are faces of 
type \II cones. There are two possibilities, according to 
Theorem~\ref{teo:incidencias}. Either two edges of the triangle are type \I 
cones, the edges correspond to the
cones $\{i,j,p\}$, $\{i,j,0\}$, $[\{0,p\},\{i,j\}\}]$, and the vertices of the 
triangle to the subdivisions $[0,i,p]$, $[0,j,p]$, $[0,i,j,p]$. These are 
$\binom{p-1}{2}$ triangles that correspond to the quadrangles of $N_0$ that 
disappear. The other triangles are those that consist on only type \II cells. 
Their edges are of the form $[\{0,p\},\{i,j\}]$, $[\{0,p\},\{i,k\}]$, 
$[\{0,p\},\{j,k\}]$ and the vertices correspond to triangulations $[0,i,p]$, 
$[0,j,p]$ and $[0,k,p]$. These are completely new triangles and the 
corresponding faces in $\disc{p}{p}$ that lie in the cell dual to $V_0$ in 
$N_0$. There are $\binom{p-1}{3}$ such triangles. Hence, there are 
$\binom{p-1}{2}+\binom{p-1}{3}=\binom{p}{3}$ triangles in $N_p$. Note that 
there are no new vertices in $N_p$, by the analysis of the edges of the 
triangle. The vertex of $N_p$ corresponding to the triangulation $[0,i,p]$ 
correspond to the (dehomogenized, closed) cone $a_0=a_p\geq a_i\geq a_j$, 
$j\notin\{i,0,p\}$.
\end{proof}

\begin{remark}
The Theorem suggests that reducing the discriminant mod $p$, in terms of the 
Newton polytope, consists in just erasing the monomials corresponding to 
triangulations with a segment of length multiple of $p$. This is not the case. 
For instance, let $n=5$, and compare $N_{0,5}$ and $N_{3,5}$ the Newton 
polytopes of the discriminant of degree 5 in characteristics $0$ and $3$. Then 
$N_{0,5}-N_{3,5}= \{(0, 2, 4, 0, 0, 2)$, $(2, 0, 0, 4, 2, 0)$, $(2, 0, 0, 5, 0, 
1)$, $(0, 4, 0, 0, 4, 0)$, $(1, 0, 5, 0, 0, 2)\}$ which are precisely the 
triangulations with a segment of length 3. But $N_{3,5}-N_{0,5}=\{(1, 3, 1, 0, 
0, 3), (0, 3, 0, 0, 1, 3, 1)\}$ are new vertices.
\end{remark}

\subsection{Universally singular polynomials}

In this section, we analyze which polynomials have a multiple tropical root at 
$0$ independently on the characteristic. That is, 
$\planouniv{n}=\cap_{p}\planopn\cap \plano{0}{n}$. We have the following 
result.

\begin{theorem}\label{teo:univ_singular}
$\planouniv{n}$ is a rational polyhedral fan of codimension $3$ in 
$\R^n$. The cones of codimension 3 consist on polynomials such that the minimum 
is attained at three different monomials $i,j,k$ such that
\begin{itemize}
 \item $d=\rad(j-i)=\rad(k-i)=\rad(k-j)$.
 \item If we eliminate all the monomials $l\equiv i\mod d$ in $f$, then the 
minimum is attained in two different monomials $r$, $s$ such that $[r-i], 
[s-i]\in \mathbb{Z}/(d)^*$ are units $\mod d$.
\end{itemize}
\end{theorem}
\begin{proof}
First, note that if $p>n+1$ is a prime, then the Euler derivative in 
characteristic $p$ coincides with one Euler derivative in characteristic zero, 
so $\plano{0}{n}=\planopn$. This means that $\planouniv{n}$ is the 
intersection of finitely many rational polyhedral fans in $\R^{n+1}$. So it is 
a rational polyhedral fan. Now, note that $\plano{0}{n}$ only contains facets 
of type \I, while $\plano{2}{n}$ contains only facets of type \III. This means 
that $\planouniv{n}$ does not contain any cell of codimension 2. Consider a 
polynomial $f$ satisfying the hypothesis. It is clear that the corresponding 
cell is of codimension three, since we have only three conditions 
$a_i=a_j=a_j<a_r=a_s$. It is also clear that $f\in \plano{0}{n}$. Now, 
$(k-i)=(k-j)+(j-i)$ and have the same radical $d$. In this case, $d$ must be 
even and $i,j,k$ must be of the same parity and $r-i,s-i$ must be odd. This 
means that $f\in \plano{2}{n}$. Next, let $p$ be a prime not dividing $d$. 
Taking an Euler derivative in characteristic $p$ can only erase at most one of 
the monomials $\{i,j,k\}$ in the support and the minimum in the derivative will 
be attained in the other two monomials. It follows that $f\in \planopn$. 
Finally, let $q$ be a prime dividing $d$. If we take a derivative $\mod q$, 
then either the three monomials $\{i,j,k\}$ belong to the support of the 
derivative or we are considering $\frac{\partial f}{\partial x-i}$, but since 
$[r]\neq [i] \neq [s]\mod q$, and we are eliminating all the monomials 
congruent to $i$ mod $q$ (and, in particular all the monomials congruent to $i$ 
mod $d$). The minimum will be attained at $r$ and $s$ and $f\in \plano{q}{n}$.

Assume now that $f$ belongs to a cell in $\planouniv{n}$ of maximal dimension. 
Since $f$ is singular in characteristic zero, the minimum of the coefficients 
is attained in three different monomials $i,j,k$. Also, if there is a prime $p$ 
such that two of the classes $[i],[j],[k]$ are equal $\mod p$, then the three 
classes must be equal $\mod p$. This means that the primes dividing $j-i$, 
$k-i$, $k-j$ are the same and we have the first condition. For the second 
condition, note that the primes not dividing $d$ do not pose any problem and 
that, for any prime $q$ dividing $d$ taking the derivative $\frac{\partial 
f}{\partial x-i} \mod q$ eliminates all the monomials congruent to $i$ $\mod 
d$. For any prime $q$ dividing $d$ there are two monomials $r_q$, $s_q$ where 
the minimum is attained in $\frac{\partial f}{\partial x-i} \mod q$. But if 
these two monomial depend on $q$, then the codimension of the cell must be at 
least 4 and we are out of hypothesis. Hence, the same monomials $r$, $s$ are 
valid for every $q$. Since $[r]\neq [i]\mod q$ for any $q$ dividing $d$, then 
$[r-i]\in \mathbb{Z}/(d)^*$ and the same happens to $[s-i]$.
\end{proof}

\begin{theorem}
If $n$ is big enough, then $\planouniv{n}$ is not a pure rational polyhedral 
fan in codimension three, there are maximal faces of codimension 
$\mathcal{O}(\log(n))$.
\end{theorem}
\begin{proof}
Take $n= 2\cdot 4^k$ and consider the first $k$ primes $p_1,\ldots, p_k$. Let 
$d=\prod_{i=1}^k p_i < 4^k < n$ by \cite{Thebook}. Take the polynomial $f$ such 
that takes the value $0$ at $0,d,2d$. And, for each $1\leq i\leq k$ takes the 
value $i$ at the monomials $d/p_i, d/p_i+d$. And, for every other monomial 
takes arbitrary values bigger than $k$. It is clear by construction that the 
polynomial belongs to $\planouniv{n}$. Because, for every prime $q$ and every 
class $l$, the minimum of $\frac{\partial f}{\partial x-l} \mod q$ is attained 
in (at least) two of the monomials in $\{i,j,k\}$ or the minimum is attained in 
$d/q,d/q+d$ if $q|d$  and $[l]=[i]\mod q$. So, this is a cell of codimension 
$k+2$. It follows that we can always construct maximal cells of codimension 
$\mathcal{O}(\log(n))$.
\end{proof}

\section{A note on the $p$-adic case}\label{sec:padics}

So far we have studied only the equicharacteristic case in which the field $\K$ 
has the same characteristic as the residue field $k$. The $p$-adic case has 
also interest on its own. however, under a $p$-adic valuation, the discriminant 
is no longer a fan. Because non-zero constants appear while taking Euler 
derivatives. These tropical discriminants are polyhedral complexes of 
codimension 1 in $\R^{n+1}$. We show that the $p$-adic discriminant performs a 
\emph{kind of interpolation} between the discriminant in characteristic zero 
and characteristic $p$.

If we take a generic point in a maximal cell of the discriminant, the 
corresponding polynomial $f$ will have only one double root $a$. Let $L_0$ be 
the set of monomials where the minimum is attained at $a$ and $L_1$ the set of 
monomials where the second minimum is attained at $a$.

\begin{theorem}\label{teo:padic}
Let $\padicdisc_p$ be the discriminant of polynomials of degree $n$ under the 
$p$-adic valuation. Then
\begin{itemize}
\item On a small ball centered at the origin, the $p$-adic discriminant equals 
the discriminant in characteristic $p$, $\padicdisc_p\cap 
B(0,\epsilon)=\discn{p}\cap B(0,\epsilon)$ for $0<\epsilon<<1$.
\item In the subset of $\mathbb{R}^n$ where $L_0<<L_1$, the maximal cells in 
common of the $p$-adic discriminant and the characteristic $p$ tropical 
discriminant are precisely the cones of type \I.
\end{itemize}
\end{theorem}
\begin{proof}
Taking the Euler derivative with respect to $x-i$ of a polynomial in $p$-adics 
increases the value of the monomials $[j]=[i]$ by the $p$-adic valuation of 
$i-j$. If the coefficients are small enough, taking this derivative has the 
same 
effect as deleting the monomials for the matter of finding the minimum. Hence 
the 
first item follows.

For the second item, Note that if we have a type \I cone in the discriminant of 
characteristic $p$ then the minimum is attained at three monomials $i,j,k$ 
defining three different classes $[i],[j],[k] \mod p$. Taking Euler derivatives 
$\frac{\partial f}{\partial x-i}$, $\frac{\partial f}{\partial 
x-j}$,$\frac{\partial f}{\partial x-k}$, in the $p$-adic case do not modify the 
other two monomials there the minimum is attained, so these cells are in common 
between the two discriminants. Now, if we have a type \II cone in the 
characteristic $p$ discriminant, then the minimum is attained at two different 
monomials $i$, $j$ with $[i]$, $[j]$. Since $L_0<<L_1$, it follows that the 
minimum in the $p$-adic Euler derivative $\frac{\partial f}{\partial x-i}$ is 
attained in the monomial $j$ alone. So, it is not a cell in the $p$-adic 
discriminant.
\end{proof}

\section*{Acknowledgements}
The author is supported by the Spanish ``Ministerio de Ciencia e Innovaci\'on'' 
and ``European Regional Development Fund" (FEDER) under the research project 
MTM2011-25816-C02-02.

\end{document}